\documentclass[12pt]{amsart}
\usepackage{amssymb}
\usepackage{amsmath, amscd}
\usepackage{amsthm}
\usepackage{xcolor}

\newtheorem{theorem}{Theorem}[section]

\newtheorem{proposition}[theorem]{Proposition}
\newtheorem{lemma}[theorem]{Lemma}

\newtheorem{corollary}[theorem]{Corollary}
\theoremstyle{definition}

\newtheorem{definition}[theorem]{Definition}

\newtheorem{remark}[theorem]{Remark}

\newcommand{\Z}{\mathbb{Z}}
\newcommand{\Q}{\mathbb{Q}}

\newcommand{\qc}{\Z(p^\infty)}

\newcommand{\Ker}{{\rm Ker}}

\begin{document}
	
\author[P. Danchev]{Peter Danchev}
\address{Institute of Mathematics and Informatics, Bulgarian Academy of Sciences, 1113 Sofia, Bulgaria}
\email{danchev@math.bas.bg; pvdanchev@yahoo.com}
\author[B. Goldsmith]{Brendan Goldsmith}
\address{Technological University, Dublin, Dublin 7, Ireland}
\email{brendan.goldsmith@TUDublin.ie; brendangoldsmith49@gmail.com}
	
\vskip1.0pc
	
\title[Super Bassian and Nearly Generalized Bassian Groups]{Supper Bassian and Nearly Generalized Bassian \\ Abelian Groups}
\keywords{Abelian groups, Bassian groups, supper Bassian groups, generalized Bassian groups, (hereditarily, nearly, super) generalized Bassian groups, Hopfian and co-Hopfian groups}
\subjclass[2010]{20K10, 20K20, 20K21}
	
\begin{abstract} In connection to two recent publications of ours in Arch. Math. Basel (2021) and Acta Math. Hung. (2022), respectively, and in regard to the results obtained in Arch. Math. Basel (2012), we have the motivation to study the {\it near} property of both Bassian and generalized Bassian groups. Concretely, we prove that if an arbitrary reduced group has the property that all of its proper subgroups are generalized Bassian, then it is also generalized Bassian itself, that is, if $G$ is an arbitrary nearly generalized Bassian group, then either $G$ is a quasi-cyclic group or $G$ is generalized Bassian of finite torsion-free rank. Moreover, we show that there is a complete description of the so-called {\it super Bassian} groups, that are those groups whose epimorphic images are all Basian, and give their full characterization. Also, we establish that the {\it hereditary} property of Bassian groups gives nothing new as it coincides with the ordinary Bassian property.
\end{abstract}
	
\maketitle
	
\vskip1.0pc
	
\section{Introduction and Main Facts}\label{intro}

Throughout the current paper, the term group will mean an additively written Abelian group. Our notation and terminology are at all standard and may be found in the books \cite{F1,F2,F3}. We will make extensive use of the notions of torsion-free rank and $p$-rank by using the corresponding notations $r_0(G)$ and $r_p(G)$, respectively as full details of these notions are given in \cite[Sections 16,35]{F1} or \cite[Sections 4.4, 5.6]{F3}. If $G$ is a mixed group, we will denote its torsion subgroup by $T(G)$ (and often we will simplify this abbreviation to $T$ if there is no possibility of ambiguity). The $p$-primary components of such a mixed group $G$ will be denote by $T_p(G)$ or $T_p$ when this is more convenient. We also adopt one slightly non-standard piece of terminology thus: a group $G$ is said to be a {\it genuine mixed group} if $G$ has both non-trivial torsion and torsion-free elements and $G$ is not splitting, i.e., in other words, the torsion subgroup of $G$ is {\it not} a direct summand of $G$. Thus, the groups that are not genuinely mixed fall into one of the three categories as follows: they are either torsion, torsion-free or splitting mixed. We also use the concept {\it an elementary $p$-group}, or sometimes {\it a $p$-elementary subgroup}, to denote a (possibly infinite) direct sum of cyclic groups of order $p$ and the name {\it elementary group} is used to denote a torsion group such that all its $p$-primary components are elementary $p$-groups.

Any other specific concepts will be explained as needed in what follows. For instance, let us recall that if $(\mathcal{P})$ is a property of groups, then a group $G$ is said to be {\it hereditarily $(\mathcal{P})$} if $G$ and all its subgroups have together the property $(\mathcal{P})$. In this section we wish to investigate
a weak variant of this concept. For simplicity of notation, we say that a group $G$ is {\it nearly $(\mathcal{P})$} if every proper subgroup of $G$ has property $(\mathcal{P})$; it is certainly {\it not} assumed a priory that $G$, itself, has the property $(\mathcal{P})$. Specifically, we intend to investigate which groups are nearly $(\mathcal{P})$ when $(\mathcal{P})$ is the property of being Bassian or generalized Bassian, respectively. Additionally, we shall examine the "super" property $(\mathcal{P})$ of Bassian groups in the sense that each
epimorphic image of the group is Bassian; this is certainly a stronger requirement than that of being Bassian. Thus, by what we have noted, our basic motivating tool in writing up the present article is to expand somewhat the corresponding results for Hopfian and co-Hopfian groups from \cite{GG}. We achieve that in Corollary~\ref{2}, Proposition~\ref{3}, Theorem~\ref{2new} and Theorems~\ref{DKdecomp}, \ref{chief}, respectively.

\medskip

The following two basic notions, that are in the focus of our present interest, appeared in \cite{CDG1} and \cite{CDG2}, respectively.

\medskip

An Abelian group $G$ is called {\it Bassian} if the existence of an injective homomorphism $G\to G/H$ for some subgroup $H$ of $G$ implies that $H=\{0\}$ and, more generally, that $G$ is {\it generalized Bassian} if the existence of an injective homomorphism $G\to G/H$ for some subgroup $H$ of $G$ implies that $H$ is a direct summand of $G$.

\medskip

Unfortunately, as already commented above, there is no a full description of the generalized Bassian groups so far. However, the following complete characterization of Bassian groups was obtained in \cite[Main Theorem]{CDG1}:

\medskip

{\bf Theorem} (Bassian). (i) {\it A reduced Abelian group $G$ is Bassian if, and only if, all of the ranks $r_0(G), r_p(G)$ (for $p$ a prime) are finite};

(ii) {\it A non-reduced Abelian group $G$ is Bassian if, and only if, it has the form $G=D\oplus R$, where $D$ is a finite dimensional $\Q$-vector space and $R$ is a reduced Bassian group}.

\medskip

It follows at once from these two necessary and sufficient conditions that a Bassian group is necessarily countable and that the finite direct sum of Bassian groups is again a Bassian group. Note also that a consequence of this classification is that the Bassian groups are hereditarily (see, e.g., \cite{GG} for both Hopfian and co-Hopfian groups): a subgroup of a Bassian group is necessarily Bassian. The converse is definitely {\it not} true since the quasi-cyclic groups $\qc$ are examples of non-Bassian groups all of whose proper subgroups are Bassian.

\medskip

Furthermore, a detailed analysis of the fundamental properties of generalized Bassian groups was given in \cite{CDG2} and \cite{DK}, respectively, although a full characterization result was {\it not} obtained there. In fact, it is {\it not} clarified yet whether a subgroup of a generalized Bassian groups is again generalized Bassian, and even finding a suitable approach how to attack this question is still problematic. However, some further advantage in this matter was made in the latter paper. Concretely, it was shown in \cite{DK} that {\it any generalized Bassian group must have finite torsion-free rank} (thus answering in the negative of a recent question posed in \cite{CDG2}), that {\it any generalized Bassian group is a direct sum of a Bassian group and an elementary group} and that {\it every subgroup of a generalized Bassian group is also a direct sum of a Bassian group and an elementary group}, as well as it was conjectured there that {\it the direct sum of a Bassian group and an elementary group has to be a generalized Bassian group} (compare also with our final Problem 3).

\medskip

\section{The Hereditary, Near and Super Bassian Properties}\label{heredit}

As mentioned in the Introduction, a group with a property $(\mathcal{P})$ is said to have the {\it hereditary $(\mathcal{P})$ property}, provided the whole group along with all of its proper subgroups possess the same property $(\mathcal{P})$; a slight variation of this concept, which we have named {\it nearly} is the following: group is said to have the {\it near $(\mathcal{P})$ property}, provided each of its proper subgroup possesses property $(\mathcal{P})$. We are  {\it not} assuming that a group which is nearly $(\mathcal{P})$ necessarily has the property $(\mathcal{P})$.

\subsection{Hereditary Bassian group}

First of all, we show that Bassian groups are, in fact, a familiar class of groups involving a hereditary condition.

\begin{proposition}\label{1} A  group is Bassian if, and only if, it is a hereditarily Hopfian group.
\end{proposition}

\begin{proof} If $G$ is a Bassian group, then, as observed in \cite{CDG1}, every subgroup of $G$ is also Bassian. Furthermore, since Bassian groups are also Hopfian, $G$ is then hereditarily Hopfian.
	
Conversely, if $G$ is a hereditarily Hopfian group, then, by the classification given in \cite{GG}, $G$ is an extension of a direct sum of finite primary groups by a finite rank torsion-free group. It thus follows from the characterization in \cite{CDG1} that $G$ is Bassian, as asserted.
\end{proof}

As an immediate consequence, we have the following surprising claim.

\begin{corollary}\label{2} A reduced group is hereditarily Bassian if, and only if, it is Bassian.
\end{corollary}

\subsection{Nearly Bassian groups}

The next result gives us a complete characterization of nearly Bassian groups.

\begin{proposition}\label{3} An arbitrary group $G$ is nearly Bassian if, and only if, either $G$ is the quasi-cyclic group $\mathbb{Z}(p^{\infty})$ for some prime $p$, or $G$ is Bassian.
\end{proposition}

\begin{proof} The sufficiency is straightforward: indeed, a proper subgroup of a quasi-cyclic group is finite and hence Bassian, while a subgroup of a Bassian group is again Bassian -- see, e.g., the discussion in the Introduction of \cite{CDG2}.
	
For the necessity, let $T$ denote the torsion subgroup of $G$ and $d(T)$ the maximal divisible subgroup of $T$. Clearly, $d(T)$ is a summand of $G$, so that $G=d(T) \oplus X$ for some subgroup $X$ of $G$. We consider three possible cases:
	
{\it Case} (i): $X = \{0\}$. Then $G$ is of the form $G = \bigoplus\limits_{p \in I} D_p$, where each $D_p$ is a divisible $p$-group and $I$ is a set of primes. If $\mid I\mid ~ > 1$, then $G$ has a proper subgroup $D_q$ for some prime $q$, which is, by hypothesis, Bassian -- contrary to \cite[Proposition 3.1]{CDG1}. We conclude that $G$ must be a divisible $p$-group for some prime $p$, and so an identical argument shows that it must have rank $1$. Thus, in Case (i), we obtain that $G \cong \mathbb{Z}(p^{\infty})$ for some prime $p$, as asserted.
	
{\it Case} (ii): $X \neq \{0\} \neq d(T)$. In this case, $d(T)$ is a proper subgroup of $G$ and hence Bassian - it is, however, impossible as observed in Case (i). So, this case cannot occur.
	
{\it Case} (iii): $d(T)=\{0\}$, that is, $G = X$. We claim that $G$ must be Bassian; for if not, then it follows from the classification of Bassian groups in \cite{CDG1} that at least one of the ranks $r_0(G), r_p(G)$ (for a prime $p$) is infinite. If any $r_p(G)$ is infinite, then the $p$-primary component of $G$, and hence $G$ itself, has a proper subgroup, $H$ say, with $r_p(H)$ infinite; but this is nonsense since $H$ would then also be Bassian. A similar argument shows that, if $r_0(G)$ is infinite, then again there is a proper subgroup $H$ of $G$ with $r_0(H)$ infinite, which is also untrue because $H$ would also be Bassian. Thus, in Case (iii), $G$ is necessarily Bassian, as required.
\end{proof}

\subsection{Super Bassian Groups}

In this subsection we investigate a property that is stronger than the property of being Bassian. In fact, our objective is to determine those groups having the property that every epimorphic image of the group is Bassian; such groups are, of course, Bassian but the converse manifestly fails: a non-reduced Bassian group will always have an epimorphic image which is a quasi-cyclic group, $\qc$, and the latter is not Bassian - see the Main Theorem of \cite{CDG1} quoted in the Introduction. However, a group which is reduced may have a divisible epimorphic image: indeed, Corner has constructed, for each positive integer $n \geq 2$, a torsion-free group of rank $n$ such that all its subgroups of rank $n-1$ are cyclic while the torsion-free quotient groups of rank $1$ are all divisible -- see, for example, \cite[Exercise 5, Section 4, Chapter 12]{F3} and, of course, every reduced unbounded $p$-group $G$ has a divisible epimorphic image of the form $G/B$, where $B$ is a basic subgroup of $G$.

Thus, motivated by these facts, we will make use of the following, somewhat {\it ad hoc} definition: A group $G$ is said to have property $(\frak{P})$ if no quasi-cyclic group is an epimorphic image of $G$.

In keeping with standard terminology used in relation to Hopfian groups (see, e.g., \cite{GG}), we make the following definition.

\begin{definition} A group $G$ is said to be {\it super Bassian} if every epimorphic image of $G$ is Bassian.
\end{definition}

Evidently, a super Bassian group is Bassian and it follows from our discussion above that a super Bassian group is necessarily reduced and also that $G$ has property $(\frak{P})$.

The classification we are seeking is given by our next result.

\begin{theorem}\label{2new} Let $G$ be an arbitrary group. The following three statements hold:
	
(i) if $G$ is torsion, then $G$ is super Bassian if, and only if, $G$ is reduced Bassian;
	
(ii) if $G$ is torsion-free, then $G$ is super Bassian if, and only if, $G$ is Bassian and has property $(\frak{P})$;
	
(iii) if $G$ is mixed, then $G$ is super Bassian if, and only if, $G$ is Bassian and has property $(\frak{P})$.
\end{theorem}

\begin{proof} Observe that in each of the three cases to be considered necessity holds obviously, so we need only give the arguments guaranteeing sufficiency.
	
{\it Case} (i): As $G$ is reduced torsion and Bassian, then it has a primary decomposition of the form $G = \bigoplus\limits_{p \ {\rm prime}}G_p$ and each $p$-primary component is a finite $p$-group. But then every epimorphic image of $G$ must have a similar primary decomposition and thus is Bassian by the classification of Bassian groups. So, $G$ is super-Bassian.
	
{\it Case} (ii): Let $G$ be a torsion-free Bassian group satisfying property $(\frak{P})$. Then $G$ is reduced and of finite torsion-free rank. Let $\phi: G \twoheadrightarrow X$ be an arbitrary epimorphism from $G$ onto $X$. We consider the three possibilities for $X$: (a) $X$ is torsion; (b) $X$ is torsion-free; (c) $X$ is mixed.  	
	
{\it Subcase} (a): As $X$ is torsion, $X$ must be reduced since otherwise it would have a quasi-cyclic summand, $Y$ say, which is then an epimorphic image of $G$ contrary to our hypothesis that $G$ has property $\frak{P}$. Let $X$ have a primary decomposition $X = \bigoplus_p X_p$ and consider a surjection $\psi_p:G \twoheadrightarrow X_p$, where $\psi_p$ is the composition of $\phi$ with the canonical projection of $X$ onto $X_p$. If each $X_p$ is bounded then $X_p$ is finite since it is then a bounded image of a torsion-free group of finite rank. We claim all the components $X_p$ must be bounded. Suppose, for a contradiction, that some component $X_p$ is unbounded and let $B_p$ be a basic subgroup of $X_P$. Now the composition of $\psi_p$ with the canonical projection of $X_p$ onto $X_p/pX_p$ is a surjection and, as $X_p/pX_p$ is bounded, we have, as observed above, that $X_p/pX_p$ is finite. This is impossible since $X_p/pX_p \cong B_p/pB_p$ would then force $B_p/pB_p$ to be finite, contrary to our assumption that $X_p$ is unbounded. Hence, $X$ has the property that each of its primary components is finite and so $X$ is Bassian and reduced.

{\it Subcase} (b): If $X$ is torsion-free, then clearly $X$ is also torsion-free of finite rank and so Bassian.

{\it Subcase} (c): If $X$ is mixed then it must be reduced and $r_0(X)$ is finite. So, $X$ is an extension of the torsion group $T(X)$ by a torsion-free group of finite rank; in particular, $X$ is an extension of a torsion group by a countable torsion-free group. Let $T(X) = \bigoplus_p X_p$ be the primary decomposition of the torsion subgroup $T(X)$ of $X$ and let $B_p$ be a basic subgroup of $X_p$ for each prime $p$ in the decomposition. Then, by an unpublished result of Corner, $B_p$ is an epimorphic image of $X_p$, say $\eta_p :X_p \twoheadrightarrow B_p$. (The result of Corner has been used previously in \cite{GG} and \cite[Theorem 4.1]{GG1}; some details of his arguments may be found in Section 4 of the latter.) Then, the composition of $\eta_p$ with the surjections of $G$ onto $X_p$ yields an epimorphism of $G$ onto the torsion reduced group $B_p$. The argument in part (i) tells us that each $B_p$ is finite, whence each $X_p$ is also finite. Thus, for each prime $p$, we have that $r_p(X)$ is finite and, as $r_0(X)$ is also finite, we conclude that $X$ is Bassian.

Combining the three subcases we see that in case (ii), each epimorphic image $X$ of $G$ is Bassian, so $G$ is super Bassian, as claimed.

{\it Case} (iii): Here we are assuming that $G$ is mixed Bassian and has property $(\frak{P})$. Let $G \twoheadrightarrow X$ be an arbitrary epimorphism. We again consider the three subcases (d) $X$ is torsion; (e) $X$ is torsion-free; (f) $X$ is mixed.

{\it Subcase} (d): As before, $X$ is necessarily reduced. Set $Z = \phi(T(G)) \leq X$. Since each $p$-primary component of $T(G)$ is finite, the same holds for $Z$ and hence $Z$ is Bassian. Furthermore, $X/Z = \phi(G)/\phi(T(G))$ is an epimorphic image of $G/T(G)$, a torsion-free group of finite rank. Therefore, $G/T(G)$ satisfies property $(\frak{P})$ since, by assumption $G$ does. Thus, by what we established in part (ii) above, the group $X/Z$ is Bassian. So, we have an exact sequence $$0 \to Z \to X \to X/Z \to 0,$$ where $Z$ is torsion Bassian and $X/Z$ is also Bassian. It now follows from Lemma~\ref{3} below that $X$ is Bassian, as asserted.

{\it Subcase} (e): If $G$ is mixed Bassian, then $r_0(G)$ is finite and thus $X$, as being a torsion-free epimorphic image of $G$, is also of finite rank and Bassian.

{\it Subcase} (f): Again we have that $r_0(X)$ is finite and if $Y = \phi(T(G))$, then $Y$ is torsion Bassian. Exactly as in subcase (d), we have that $X/Y = \phi(G)/\phi(T(G))$ is an epimorphic image of the finite rank torsion-free group $G/T(G)$. Observe that $G/T(G)$ has property $(\frak{P})$ since $G$ has this property. It now follows from case (ii) that $X/Y$ is Bassian, so that $X$ is an extension of the torsion Bassian group $Y$ by the Bassian group $X/Y$. Applying Lemma~\ref{3} stated below, we see that $X$ is also Bassian, as wanted.

Combining the subcases (d),(e) and (f), we have established that if $G$ is mixed Bassian and has property $(\frak{P})$, then $G$ is super Bassian, as desired.
\end{proof}

It remains only to establish the cited above Lemma~\ref{3}, which is a simple extension of \cite[Proposition 2.3]{CDG1}.

\begin{lemma}\label{3} If $0 \to B \to G \to C \to 0$ is an extension of a torsion Bassian group $B$ by a Bassian group $C$, then $G$ is Bassian.
\end{lemma}

\begin{proof} Choose $H \leq G$ such that $H/B = T(C)$, the torsion subgroup of $C$. Then, we have an exact sequence
$$0 \to H \to G \to G/H \cong C/T(C) \to 0.$$ Since $C$ is Bassian, the factor-group $C/T(C)$ is torsion-free of finite rank and so, in view of \cite[Proposition 2.3]{CDG1}, it suffices to show that $H$ is Bassian since it is certainly torsion. To that goal, let $H = \bigoplus H_p$ be the primary decomposition of $H$ and note that, if $B_p$ is the corresponding $p$-primary component of $B$, then $B_p \leq H_p$ and $H_p/B_p \cong C_p$, where the latter is the $p$-primary component of $T(C)$. Since $B, T(C)$ are both Bassian, the subgroups $B_p, C_p$ are finite giving that each primary component of $H$ is also finite and so $H$ is Bassian, as expected.
\end{proof}

The requirement in Theorem~\ref{2} above that $G$ have property ($\frak{P}$) is a little unsatisfactory, because it is not immediate how to determine this from the internal structure of $G$. Fortunately, when a group is Bassian, we have the following helpful equivalence.

\begin{proposition} Let $G$ be a Bassian group, then $G$ has property ($\frak{P}$) if, and only if, the torsion-free quotient $G/T(G)$ has property ($\frak{P}$).
\end{proposition}

\begin{proof} If $G$ has property ($\frak{P}$), then it is clear that $G/T(G)$ has the same property even when $G$ is not Bassian.
	
Conversely, assume the contrary that $G$ does not have property ($\frak{P}$), and let $\phi: G \twoheadrightarrow \qc$ be a surjection onto some quasi-cyclic group $\qc$. If $T_p$ denotes the $p$-primary component of $G$, then $\phi(T(G)) = \phi(T_p)$ is finite since $G$ is Bassian. Thus, $\phi(T(G)) = C$, where $C$ is a finite cyclic subgroup of $\qc$. Hence, $$\phi(G)/\phi(T(G)) = \qc/C \cong \qc.$$ But $\phi(G)/\phi(T(G))$ is an epic image of $G/T(G)$, so that there is an epimorphism $G/T(G) \twoheadrightarrow \qc$ and, consequently, $G/T(G)$ does not have property ($\frak{P}$), as required.
\end{proof}

It is also worthwhile noticing that, for a finite rank torsion-free group $G$, the following are equivalent (see \cite{DK} as well):

\medskip

(a) $G$ has no homomorphic images that are quasi-cyclic, i.e. $G$ has property $(\frak{P})$;

(b) $F$ is a free subgroup of $G$ such that $G/F$ has finite $p$-torsion for all primes $p$;

(c) For every free subgroup $F$ of $G$ such that $G/F$ is torsion, $G/F$ has finite $p$-torsion for all primes $p$;

(d) For every prime $p$, the localization $G_{(p)}$ is a free $\mathbb{Z}_{(p)}$-module (i.e., it is locally free).

\medskip

We end our work in this section with the following interesting question:

\medskip

\noindent{\bf Problem 1}. Describe the structure of hereditary generalized Bassian groups and super generalized Bassian groups by finding a complete structural characterization of them.

\section{Nearly Generalized Bassian Groups}\label{ngbg}

In this section, we consider the problem analogous to that addressed in Proposition \ref{3} in Section \ref{heredit}, where Bassian groups are replaced by generalized Bassian groups. Thus, we shall say that a group $G$ is {\it nearly generalized Bassian} if every proper subgroup of $G$ is generalized Bassian. Before proceeding to consideration of this problem, we present a brief review of known material about generalized Bassian groups. Fuller details and proofs may be found in \cite{CDG2} and \cite{DK}, respectively.

\medskip

$\bullet$ A generalized Bassian group has finite torsion-free rank.

\medskip

Since this fact is very important for our further presentation, we shall give an independent and more conceptual verification by modifying some of the arguments used in \cite{CDG2}.

To that purpose, assume that $G$ is generalized Bassian and is not torsion. If $G$ is torsion-free, then it necessarily has finite rank by \cite[Proposition 2.8]{CDG2}. If $G$ is splitting mixed, then again it is of finite torsion-free rank by \cite[Proposition 2.11]{CDG2}.
 	
So, consider the case where $G$ is genuine mixed and assume, for a contradiction, that $r_0(G) = \kappa$ is infinite.
However, we know from \cite[Proposition 3.3]{CDG2} that if $T$ is the torsion subgroup of $G$, then $\kappa <\mid T\mid$. Furthermore, $T = \bigoplus_p T_p$, where each $T_p$ has the form $T_p E_p \oplus F_p$, with $E_p$ an elementary $p$-group and $F_p$ a finite $p$-group. It follows from Remark 3.4 in \cite{CDG2} that $\mid T_p\mid > r_0(G)$ for infinitely many primes $p$, so that infinitely many of the $E_p$ have $p$-rank $> \kappa$.
 	
Since $G/T$ is torsion-free of infinite rank $\kappa$, there is a subgroup $H$ of $G$ with $G = H + T$ and $\mid H\mid = \kappa$. Set $B = H + \bigoplus+p  F_p$, so that $\mid B\mid$ is also $\kappa$. Now, $\bigoplus_p F_p \leq T \cap B \leq T = \bigoplus_p F_p \oplus \bigoplus E_p$, so there is a subgroup $D \leq \bigoplus_p E_p$ with $T = (T \cap B) \oplus D$. But $B \cap D \leq T \cap B \cap D = \{0\}$ and so $$G = H + T \leq B + T = B + (T \cap B) + D \leq G,$$ hence $G = B + D = B \oplus D$.
 	
Since $G$ is generalized Bassian and $B$ is a summand of $G$, then $B$ is also generalized Bassian utilizing \cite[Lemma 2.3]{CDG2}. Note that, as $D$ is torsion, $r_0(G) = r_0(B) = \kappa$. However, $B = H + \bigoplus_p F_p$, so our assumption that $\kappa$ is infinite yields $\mid B\mid = \mid H\mid = \kappa$, and thus it follows that the torsion subgroup $t(B)$ of $B$ must have cardinality $\leq \kappa$. This means that $\mid t(B)\mid \leq \kappa = r_0(B)$ contrary to \cite[Proposition 3.3]{CDG2}. This contradiction forces at once $r_0(G)$ to be finite when $G$ is genuine mixed, thereby establishing the result after all, as asked for.

\medskip

$\bullet$ If $G$ is generalized Bassian, then the torsion subgroup  $T$ of $G$ is of the form $T = \bigoplus_p T_p$ where each $T_p$ is a $p$-primary group of the form $T_p = F_p \oplus E_p$ where $F_p$ is a finite $p$-group and $E_p$ is an elementary $p$-group.

\medskip

We note now a simple result that gives an initial insight into the structure of generalized Bassian groups.

\begin{proposition} If $G$ is a genuinely mixed group which is generalized Bassian, then $G$ is an extension of an elementary group by a Bassian group.
\end{proposition}

\begin{proof}
As noted above, each $p$-primary component of $G$ is of the form $T_p(G)  = F_p \oplus E_p$, where $F_p$ is a finite $p$-group and $E_p$ is an elementary $p$-group. Furthermore, $r_0(G)$ is finite. Let $E$ be the direct sum over the various primes $p$ of the elementary groups $E_p$, and let $F$ be the corresponding direct sum of the finite $p$-groups $F_p$; thus the torsion subgroup $T$ of $G$ has the form $T = E \oplus F$.
	
Consider now the group $G/E$: we have an exact sequence $$ 0 \to T/E \to G/E \to G/T \to 0.$$ Since $T/E \cong \bigoplus_p F_p$ and each $F_p$ is a finite $p$-group, we get that $T/E$ is a torsion Bassian group. Furthermore, $G/T$ is torsion-free of finite rank, and hence is Bassian. It follows from \cite[Proposition 2.3]{CDG2} that $G/E$ is Bassian. Thus, $G$ is an extension of the elementary group $E$ by the Bassian group $G/E$, as stated.
\end{proof}

In fact, Danchev and Keef have shown in \cite{DK} the more general fact that a generalized Bassian group is actually a split extension of an elementary group by a Bassian group. Their arguments are developed in the context of their deep investigations into mixed Abelian groups with bounded $p$-torsion. We present here another version of this result based on the results established in \cite{CDG2}.

Recall that in \cite{CDG2}, the class $\mathcal{P}$ of groups $G$ was considered, where $G \in \mathcal{P}$ if there
is an infinite set of primes $\Pi$ with $T \leq G \leq \bar{T}$, where $T$ is the torsion subgroup of $G$ and $T$ is a direct sum of $p$-primary bounded components $T_p$, $T = \bigoplus_{p \in \Pi} T_p$ and $\bar{T} = \prod_{p \in \Pi} T_p$. In \cite[Corollary 3.14]{CDG2} a decomposition of a generalized Bassian group in $\mathcal{P}$ was given: $G = A \oplus H$, where $A$ is torsion and each $p$-primary component of $A$ is a direct sum of an elementary $p$-group and a finite $p$-group, while $H$ is a Bassian group. An examination of the proof shows that the generalized Bassian property od $G$ was used only for two purposes (i) to show that $r_0(G)$ was finite and (ii) that $A$ had the desired decomposition with primary components of the form \lq\lq elementary plus finite\rq\rq. Thus, in fact, an extended version of \cite[Corollary 3.14]{CDG2} holds as follows.

\begin{proposition}\label{elplusb} Suppose that $G \in \mathcal{P}$ with $r_0(G)$ finite, and each $p$-primary component $T_p$ of $G$ has the form $T_p = E_p \oplus S_p$, where $E_p$ is an elementary $p$-group and $S_p$ is a finite $p$-group. Then $G$ has a decomposition $G = A \oplus H$ with $A$ torsion and having primary components of the form $E'_p \oplus F_p$ with $E'_p$ elementary and $F_p$ finite, and $H$ Bassian.
\end{proposition}

Now, consider an arbitrary genuinely mixed group $G$ which is generalized Bassian. Then, certainly each primary component of $G$ has the form \lq\lq elementary plus finite\rq\rq \ , so that the torsion subgroup of $G$ has the form $T = E \oplus S$, with $E$ elementary and $S$ a direct sum over primes $p$ of finite $p$-groups. Therefore, we have an exact sequence $$0 \to T \to G \to G/T \leq D \to 0,$$ where $D$ is a finite dimensional $\Q$-space. Applying Proposition 24.6 in \cite{F1} (see also Exercise (6) in \cite{F3}), we obtain a group $G'$ with $$0 \to T \to G' \to D \to 0;$$ note that $G' \in \mathcal{P}$ and has the same torsion subgroup as $G$. Thus, Proposition \ref{elplusb} enables us to get a decomposition $G' = A \oplus H$ with $A$ torsion and $H$ Bassian. Furthermore, $A$ has the form $A = E \oplus F$ with $E$ elementary and $F$ is a direct sum over primes $p$ of finite $p$-groups. Absorbing $F$ into $H$, we write that $G' = E \oplus H'$, where $H' = F \oplus H$. Since $H$ is Bassian and each primary component of $F$ is finite, it now follows easily that $H'$ is Bassian, as required.

Finally, observe that as $E \leq T$, modularity gives us that $$G = G \cap (E \oplus H') = E \oplus (G \cap H');$$ since Bassian groups are hereditarily Bassian - see, for example, the comment in the Introduction to \cite{CDG2} - it follows that $H_0 = G \cap H'$ is Bassian, so that $G = E \oplus H_0$ with $E$ elementary and $H_0$ Bassian.

Thus, we have established the following statement from \cite{DK}.

\begin{theorem}\label{DKdecomp} If $G$ is a genuine mixed group which is generalized Bassian, then $G$ has a decomposition of the form $G = E \oplus H$ with $E$ elementary and $H$ Bassian.
\end{theorem}

It follows from the classification of torsion, torsion-free and mixed generalized Bassian groups - see, for instance, Proposition 2.8, Theorem 2.10 and Proposition 2.11 in \cite{CDG2} - that an \lq\lq elementary plus Bassian\rq\rq \  decomposition also holds in each of these three cases. So, we derive the following consequence.

\begin{corollary} If $G$ is an arbitrary generalized Bassian group, then $G$ has a decomposition $G = E \oplus H$, where $E$ is elementary and $H$ is Bassian.
\end{corollary}

We return now to consideration of groups which are nearly generalized Bassian. Such groups are, of course, of finite torsion-free rank. (We remark that this does not require the full force of the result of Danchev and Keef mentioned in the first bullet point above: if a nearly generalized Bassian group is not of finite torsion-free rank then it will have a free subgroup of infinite rank, but such a group is definitely {\it not} generalized Bassian by \cite[Example 2.7]{CDG2}.) In addition, if $G$ is a nearly generalized Bassian group, then we have:

\medskip

$\bullet$ If $X$ is a proper torsion subgroup of $G$, then $X$ is generalized Bassian and so $X$ is of the form $X = \bigoplus X_p$, where each $X_p$ is a $p$-group of the form $X_p = E_p \oplus F_p$ with $E_p$ elementary and $F_p$ finite.

\medskip

Our first step in classifying groups which are nearly generalized Bassian is to deal with divisible groups.

\begin{proposition}\label{divnear} Let $G$ be a divisible group which is nearly generalized Bassian. Then, the following three points are true:
	
(i) $G$ is either torsion or torsion-free;
	
(ii) if $G$ is torsion, then $G \cong \qc$ for some fixed but arbitrary prime $p$;
	
(iii) if $G$ is torsion-free, then $G$ is a finite dimensional $\Q$-vector space and $G$ is generalized Bassian (and even Bassian).
\end{proposition}

\begin{proof} (i) If $G$ is mixed and divisible, then it must have a proper subgroup which is isomorphic to $\qc$ for some prime $p$. Since $\qc$ is not generalized Bassian, no such group exists.
	
(ii) If $G$ is torsion divisible of rank greater than one, then $G$ has a proper quasi-cyclic  subgroup which would then be generalized Bassian. As observed above, this is impossible, so $G$ has rank one and thus is isomorphic to $\qc$ for some prime $p$.	
	
(iii) If $G$ is torsion-free divisible, then, as observed above, $G$ must be of finite rank, as required.
\end{proof}

It follows from an easy application of the above Proposition \ref{divnear} that if $G$ is a nearly generalized Bassian group which is neither divisible nor reduced, then $G$ has the form $G = \bigoplus_n \Q \oplus G'$, where $0 \neq n$ is finite and $G'$ is a reduced nearly generalized Bassian group.

Our next result clarifies the situation for reduced nearly generalized Bassian groups which are either torsion, torsion-free or splitting mixed, respectively.

\begin{proposition}\label{rednear} Let $G$ be a reduced group which is nearly Bassian. Then if $G$ is either torsion, torsion-free or splitting mixed, then $G$ is generalized Bassian.
\end{proposition}

\begin{proof} Since finite groups are always generalized Bassian (and even Bassian), we may assume that $G$ is infinite. Furthermore, we have observed above that $r_0(G)$ is finite, so if $G$ is torsion-free it is a reduced group of finite rank which is then certainly generalized Bassian.
	
If $G$ is a $p$-group, then we see that $G$ has a direct decomposition $G = \Z(p^n) \oplus H$ for some finite $n$. But then, by hypothesis, $H$ is a generalized Bassian $p$-group and hence has the form $H = E \oplus F$, where $E$ is an elementary $p$-group and $F$ is finite. But then $G$, itself, has a similar decomposition and hence, with the aid of \cite[Theorem 2.10]{CDG2}, $G$ is generalized Bassian. If, however, $G$ is torsion, then it follows easily using the primary decomposition of $G$ and the result just established for $p$-groups that $G$ is generalized Bassian, as wanted.
	
Finally, if $G$ is splitting mixed, both its torsion and torsion-free parts are reduced generalized Bassian and so $G$, itself, is then generalized Bassian by virtue of \cite[Corollary~2.12]{CDG2}, as desired.
\end{proof}

Combining the results obtained above and using the classification results in \cite{CDG2}, it follows readily that we obtain:

\begin{proposition} Let $G$ be a nearly generalized Bassian group which is either torsion, torsion-free or splitting mixed, then either $G \cong \qc$ for some prime $p$, or $G$ is generalized Bassian.
\end{proposition}

To complete the classification of nearly generalized Bassian groups, it remains to consider the situation in which $G$ is nearly generalized Bassian and genuinely mixed. For the remainder of this section we shall assume $G$ is of this type. 	

The following observations will be useful in dealing with the genuinely mixed nearly generalized Bassian groups.

\medskip

$\bullet$ If $\varphi:G \to G/N$ is an injection then, since $r_0(G)$ is finite, we have $r_0(G) \leq r_0(G/N) = r_0(G) - r_0(N)$, so that $r_0(N) = 0$ or, equivalently that $N \leq T$.

\medskip

$\bullet$ If $\varphi$ embeds $G \to G/N$, then we have that $\varphi(T) \leq T(G/N) = T/N$ since $N\leq T$. Furthermore, as we are assuming $G$ is genuinely mixed, $T$ itself is generalized Bassian, and it follows that $N$ is a direct summand of $T$, so that if $N = \bigoplus N_p$, then each primary component $N_p$ of $N$ is a direct summand of the corresponding $T_p$.

\medskip

$\bullet$ Since $\varphi(T) \leq T/N$, the full invariance of the primary components gives that $\varphi(T_p) \leq T_p/N_p$ and so $\varphi(pT_p) \leq p(T_p/N_p) = (pT_p + N_p)/N_p$. Hence, as $pT_p = pF_p$, we have that
$$pF_p = \varphi(pT_p) \leq (pF_p + N_p)/N_p \cong pF_p/(N_p \cap pF).$$
However, $pF_p$ is finite so that $N_p \cap pF_p = \{0\}$, and it now follows from Lemma~\ref{basic} below that $N_p$ is necessarily elementary.

\medskip

We are now ready to prove the promised lemma applied above.

\begin{lemma}\label{basic} Suppose that $A,B,C$ are $p$-groups with $A \leq B \oplus C$. If $B$ is elementary and $A \cap pC = \{0\}$, then $A$ is elementary.
\end{lemma}

\begin{proof} Let $x \in A$, so that $x = b + c$ with $b \in B, c \in C$. Then $px = pb + pc = pc \in A \cap pC = \{0\}$.
\end{proof}

\medskip

\begin{remark} We can actually say a bit more. When we work with a single prime $p$, then we can modify the $p$-primary component so that in the decomposition $E_p \oplus F_p$ that $F_p$ has no summands of order $p$. The resulting enlarged elementary group will continue to be a direct summand of $G$. In this situation, we can conclude that $N_p \leq E_p$: if $N_p$ is not contained in $E_p$, then $N_p \cap F_p$ is a non-zero summand of the elementary group $N_p$ and hence a summand of $T_p$. This, however, is impossible since if $\langle x\rangle$ is a summand of $N_p$, then the height $ht_{T_p}(x)$ is necessarily $0$, but any $x \in N_p \cap F_p$ must be of the form $x = py$, because $F_p$ has no summands of order $p$. In fact, it is easy to see that this procedure will work for any finite set of primes. Thus, in the genuinely mixed situation, we may assume that $N_p$ is actually a summand of $T_p$.
\end{remark}

We, thus, now come to the following key statement.

\begin{proposition}\label{major} Let $G$ be a genuinely mixed group which is nearly generalized Bassian. Then, $G$ is generalized Bassian.
\end{proposition}

\begin{proof} Let us notice the crucial fact that, if $\varphi: G \to G/N$ is an embedding for some $N \neq 0$, then from our observations above we can find a decomposition $G/N = (T_p/N_p) \oplus (G'/N')$, where $N' = \bigoplus\limits_{q \neq p}N_q$ and $G = T_p \oplus G'$.
	
So, we are now in the following situation: $\varphi$ is an embedding of $G$ into $G/N$ and we can decompose $G = T_p \oplus G'$ with $T_p \neq \{0\}$ and $N = N_p \oplus N'$, where $N' = \bigoplus\limits_{q \neq p}N_q$. Then, $G/N = (T_p/N_p) \oplus(G'/N')$. Now, let $\pi$ be the canonical projection of $G/N$ onto $G'/N'$ and set $\psi = \pi\varphi$. Consider the restriction $\psi' = \psi\upharpoonright G'$, a mapping from $G' \to G'/N'$. We claim that $\psi'$ is monic; if not, then there is a nonzero $x \in G'$ such that $\psi'(x) = 0$, so that $\varphi(x) \in \Ker \pi = T_p/N_p$, which quotient is a $p$-group. However, if $x$ is a torsion element of $G'$, then its order is relatively prime to $p$. Since $\varphi$ is monic, it preserves orders, so that $x = 0$, a contradiction. Thus, $x$ must be an element of infinite order, which maps to an element of order $p^n$ for some finite $n$. This, too, is impossible, so $\psi'$ is monic, as claimed.
	
The mapping $\pi\varphi: G' \to G'/N'$ is an embedding and, by the hypothesis posed on $G$, we know that $G'$ is generalized Bassian, whence $N'$ is a direct summand of $G$. It now follows immediately that $N$ is then a direct summand of $G$, since, as observed above, $N_p$ is a direct summand of $T_p$. Consequently, $G$ is also generalized Bassian, as asserted.
\end{proof}

Summarizing the results obtained above, we have established the following curious assertion that is our main motivating result.

\begin{theorem}\label{chief} If $G$ is an arbitrary nearly generalized Bassian group, then either $G$ is a quasi-cyclic group or $G$ is generalized Bassian.
\end{theorem}

We close our work in this section with the following two difficult questions (compare also with the queries from \cite{DK}).

\medskip

\noindent{\bf Problem 2}. If the Bassian group $B$ has no direct summands isomorphic to $\Z(p)$ for each prime $p$, is then the direct sum $E\oplus B$ a generalized Bassian group, whenever $E$ is an elementary group?

\medskip

We would like to notice that in \cite{DK} was established that {\it the direct sum of an elementary group and a Bassian group such that every its elementary summand is finite is a generalized Bassian group}.

\medskip

\noindent{\bf Problem 3}. Does it follow that the direct sum of a Bassian group and a generalized Bassian group remains a generalized Bassian group?

\section{Concluding Remarks}

Two obvious questions arise: what can be said about the hereditarily and super generalized Bassian groups (see Problem 1 quoted above)? A partial answer is easily given in relation to the first question.

If $G$ is generalized Bassian and either torsion, torsion-free or splitting mixed, then  every subgroup of $G$ is likewise generalized Bassian. By utilizing the \lq elementary plus Bassian\rq \ decomposition of Danchev-Keef given in Theorem \ref{DKdecomp} in Section \ref{ngbg}, it is clear that any subgroup of a generalized Bassian group is again the direct sum of an elementary group and a Bassian group. However, it is not clear, and it seems to be an extremely difficult question, as to how to show that such a group is generalized Bassian. However, several examples are given in \cite{DK} of sufficient conditions to ensure the generalized Bassian property holds, but no a complete result is known yet; see Conjecture 1.3 in \cite{DK}.



\medskip

\noindent{\bf Funding:} The scientific work of the first-named author (P.V. Danchev) was supported in part by the Bulgarian National Science Fund under Grant KP-06 No. 32/1 of December 07, 2019, as well as by the Junta de Andaluc\'ia under Grant FQM 264, and by the BIDEB 2221 of T\"UB\'ITAK.

\end{document}